\documentclass[10pt]{article}
\usepackage{amsmath,amssymb,amsthm}
\usepackage{color}
\usepackage{booktabs}
\usepackage{graphicx}
\usepackage{epstopdf}
\usepackage{algorithm,algorithmic}
\usepackage{verbatim}
\usepackage{url}
\usepackage{float}
\graphicspath{{./eps/}}
 \usepackage{array}
 \usepackage{multirow}
 \usepackage{algorithmic}
\usepackage{algorithm}
\newcolumntype{P}[1]{>{\raggedright\arraybackslash}p{#1}}

\theoremstyle{plain}
\newtheorem{theorem}{Theorem}[section]
\newtheorem{remark}{Remark}[section]

\newtheorem{lemma}{Lemma}[section]

\newtheorem{proposition}{Proposition}[section]

\newtheorem{example}{Example}[section]
\numberwithin{equation}{section}

\newcommand{\R}{\mathbb{R}}
\newcommand{\N}{\mathbb{N}}

\newcommand{\E}{\mathbb{E}}

\title{A modified discrepancy principle to attain optimal convergence rates under unknown noise}
\author{ Tim Jahn\thanks{Institute of Mathematics, Goethe-University Frankfurt, Germany (\texttt{jahn@math.uni-frankfurt.de})}}
\begin{document}

\maketitle

% REQUIRED
\begin{abstract}
We consider a linear ill-posed equation in the Hilbert space setting. Multiple independent unbiased measurements of the right hand side are available. A natural approach is to take the average of the measurements as an approximation of the right hand side and to estimate the data error as the inverse of the square root of the number of measurements. We calculate the optimal convergence rate (as the number of measurements tends to infinity) under classical source conditions and introduce a modified discrepancy principle, which asymptotically attains this rate.
\end{abstract}

\section{Introduction}

So we aim to solve $K\hat{x}=\hat{y}$, where $K$ is compact with dense range and $\hat{x}$ and $\hat{y}$ are elements of infinite-dimensional Hilbert spaces $\mathcal{X}$ and $\mathcal{Y}$. The exact data $\hat{y}$ is unknown, but we have access to multiple and unbiased i.i.d. measurements $Y_1,...,Y_n$ with unknown arbitrary distribution  and finite variance ($\E\|Y_1-\hat{y}\|^2<\infty$). Note that at this point the measurements are infinite-dimensional objects (e.g. functions), we will later discretise along the singular vectors of the operator $K$. Repeating and averaging the measurement process is a standard engineering practice to estimate and reduce random uncertainties, see \cite{taylor1997introduction},\cite{bevington2003data} and \cite{lyons2004understanding} for introducing monographs  on the subject of error analysis from a practical view point. In the given setting, a natural estimator of the unknown data $\hat{y}$ is the sample mean 

\begin{equation*}
\bar{Y}_n = \frac{1}{n}\sum_{i=1}^n Y_i.
\end{equation*}

The compactness of $K$ implies that the equation is ill-posed, so that one cannot rely on classical direct methods like $LR$- or $UR$-decomposition to determine the (generalised) inverse of $K$. Regularisation is needed, and the inverse is replaced with a family of related but continuous approximations, e.g. Tikhonov or spectral cut-off regularisation. The particular choice of the approximation has to be based inevitably on knowledge of an upper bound of the true error $\delta_n^{true}:=\|\bar{Y}_n-\hat{y}\|$, as the famous result of Bakushinskii \cite{bakushinskii1984remarks} states. While the exact value of $\delta_n^{true}$ is clearly not given due to randomness, its variance depends mainly on the number of measurements,

\begin{equation*}
\E\left[{\delta_n^{true}}^2\right] = \E\|\bar{Y}_n-\hat{y}\|^2 = \frac{\E\|Y_1-\hat{y}\|^2}{n}.
\end{equation*}

Thus 

\begin{equation*}
\delta_n^{est}:=\frac{1}{\sqrt{n}}\quad\mbox{or}\quad \delta_n^{est}:=\frac{\sqrt{\frac{1}{n-1}\sum_{i=1}^n \|Y_i-\bar{Y}_n\|^2}}{\sqrt{n}}
\end{equation*}

are natural estimators of the unknown true error $\delta_n^{true}$. So a natural approach for the solution of the equation is to use the mean $\bar{Y}_n$ and the estimated data error $\delta_n^{est}$ together with a deterministic regularisation method. Indeed, in \cite{harrach2020beyond} it was verified, that the approach converges in a suitable sense for a large class of regularisation methods. See also \cite{harrach2020regularising} and \cite{mika2021towards}, where this approach was extended to settings involving white or Poissonian noise.
 The rate of convergence of a given regularisation is known to depend on a certain smoothness of the unknown solution $\hat{x}$ relative to the operator $K$. Classical convergence rates for deterministic noise (i.e. in a setting where one knows an upper bound for the norm of the noise) are deduced by a worst case error analysis. In our setting however, the noise, though random, typically excludes many 'bad' directions for a fixed unknown error distribution. So as it is typical under random noise, see e.g. \cite{bauer2008regularization}, \cite{blanchard2012discrepancy} or \cite{lu2014discrepancy}, the optimal rates obtained here should be substantially better than the ones one would expect from a deterministic worst case error analysis. Indeed, we show that the optimal rates here are better than for the deterministic worst case.  The main result of this work then constitutes of a modified discrepancy principle, which yields (almost) the best possible rate for arbitrary unknown error distributions.

 Denote by $(\sigma_j,u_j,v_j)_{j\in\N}$ the singular value decomposition of $K$ (i.e. $(u_j)_{j\in\N}$ is an orthonormal basis of $\mathcal{Y}$ (note that $K$ is assumed to have dense range), $(v_j)_{j\in\N}$ an orthonormal basis of $\mathcal{N}(K)^\perp$, $(\sigma_j)_{j\in\N}$ a monotone to 0 converging sequence of positive numbers and it holds that $K v_j = \sigma_j u_j$). 
In the following we will restrict to the spectral cut-off regularisation and to mildly ill-posed problems, i.e. we assume that there exists $q>0$ such that $\sigma_j^2 \asymp j^{-q}$. Thus the reconstruction will be based on the projections $(Y_i,u_j)$ for $i,j\in\N$. The unbiasedness assumption reads $\E[(Y_i,u_j)] = (\hat{y},u_j)$ for all $i,j\in\N$ and we moreover assume that there are $C_p>0, p>1$ with $\E[\left(Y_i- \hat{y},u_j\right)^2]\le C_pj^{-p}$. Later we will always consider only a finite number of components for a fixed number of measurements $n$. Spectral cut-off at truncation level $k$ for the component wise averages then yields the following estimator for $\hat{x}$
\begin{equation}\label{int:eq0}
\bar{X}_k^n:=\sum_{j=1}^k \frac{1}{\sigma_j}\left(\bar{Y}_n, u_j\right) v_j.
\end{equation}

In order to find a reasonable reconstruction the truncation level has to be determined dependent on the (estimated) noise level, which depends on the number of measurements $n$.

The rest of the paper is organised as follows. In the following Section \ref{sec:2} we state the three main Theorems, which are proven in Section \ref{sec:3}. The main result is accompanied by numerical experiments in Section \ref{sec:4} and the article ends with a short conclusion in Section \ref{sec:5}.

\section{Main results}\label{sec:2}

We derive convergence rates with respect to classical H\"older-type source conditions 

\begin{equation}\label{int:eq111}
\hat{x} \in \mathcal{X}_{\nu,\rho}:=\{ (K^*K)^{\nu/2}\xi~ |~\xi\in \mathcal{X}, \|\xi\|\le \rho\}=\left\{ \sum_{j=1}^\infty \sigma_j^{\nu}(\xi,v_j)v_j~|~\xi \in\mathcal{X}, \|\xi\|\le \rho\right\}.
\end{equation}

If $\hat{x}\in\mathcal{X}_{\nu,\rho}$, we say that $\hat{x}$ obeys smoothness $(\nu,\rho)$ relative to $K$. Via \eqref{int:eq0} a whole class of estimators indexed by $k\in\N$ is defined, which is also known under the term projection estimators (with respect to the singular value decomposition, see \cite{cavalier2011inverse}). The first result gives the optimal error bound for our estimators \eqref{int:eq0} on $\mathcal{X}_{\nu,\rho}$, where we measure performance by the integrated mean squared error (also called the minimax ($L^2$)-risk in this context). 

\begin{theorem}\label{opt1}
Let $\hat{y}:=K\hat{x}$. Assume that $Y_1,Y_2,...$ are i.i.d.  for $i=1,2,...$ with $\E[Y_1] = \hat{y}$. Moreover, assume that there are $q>0,p>1$ with $\sigma_j^2 \asymp j^{-q}$ and $\E\left(Y_{1} - \hat{y},u_j\right)^2 \asymp j^{-p}$. Then there holds

$$ \inf_{k\ge 1} \sup_{\hat{x}\in\mathcal{X}_{\nu,\rho}} \E\|\bar{X}^n_k - \hat{x}\|^2  \asymp \begin{cases}  \frac{1}{n} & q-p<-1\\
                     \frac{\log(n \rho)}{n} & q-p=-1\\
                     \rho^\frac{q+1-p}{(\nu+1)q+1-p}\left(\frac{1}{n}\right)^\frac{\nu}{\nu+1 - \frac{p-1}{q}} & q-p>-1
                     \end{cases}.$$

In particular, for the a priori choice 

$$k_n\asymp \begin{cases} \left(\rho n\right)^{\frac{1}{\nu q}} & q-p\le -1\\
                                \left(\rho n\right)^{\frac{1}{(1+\nu)q + 1-p}} & q-p>-1\end{cases}$$

it holds that

$$ \sup_{\hat{x}\in\mathcal{X}_{\nu,\rho}} \E\|\bar{X}^n_{k_n} - \hat{x}\|^2\asymp\inf_{k \ge 1} \sup_{\hat{x}\in\mathcal{X}_{\nu,\rho}} \E\|\bar{X}^n_k - \hat{x}\|^2. $$  
 
\end{theorem}

Note that under additional assumption, one can show that the rate from Theorem \ref{opt1} is (up to a constant factor) the optimal rate for all possible estimators, not just for projection estimators \eqref{int:eq0}. See e.g. \cite{pinsker1980optimal} and \cite{belitser1994minimax} for the case where $\left((Y_1-\hat{y},u_j)\right)_{j\in\N}$ are independent and Gaussian.

In view of the fact that the optimal worst case error bound for deterministic noise level $1/\sqrt{n}$ under the source condition $\hat{x}\in\mathcal{X}_{\nu,\rho}$ has order $(1/\sqrt{n})^\frac{\nu}{\nu+1}$, we see that the minimax risk attained by the oracle $k_n$ is in all cases strictly better. In particular, for $q-p< -1$, the problem is in fact well-posed.
However, the above choice $k_n$ requires knowledge of both the smoothness $\nu$ and the decay of variances $p$.  A plain use of the discrepancy principle \cite{morozov1968error} as an adaptive strategy to determine the truncation level would be to find $k=k(\bar{Y}_n,\delta_n^{est})$, such that the size of the residual is approximately equal to the estimated noise level, i.e. by the relation

\begin{equation}\label{int:eq1}
\sqrt{\sum_{j=k+1}^\infty (\bar{Y}_n,u_j)^2} \approx \delta_n^{est}.
\end{equation}

In \cite{harrach2020beyond} it was shown, that the choice \eqref{int:eq1} adapts to the unknown smoothness $\nu$ in the sense that asymptotically the optimal deterministic bound holds with a probability converging to $1$. According to Theorem \ref{opt1}, this is suboptimal. The reason is an intrinsic drawback of the plain discrepancy principle for statistical noise, which tempts to stop too late.  We therefore consider in this work a modified version of the discrepancy principle, which also takes information about the stochastic nature of the noise into account.

We first formulate a simplified version of the main result to illustrate the approach. As already mentioned, the rate of convergence depends on certain smoothness properties of the real unknown solution $\hat{x}$ relative to the forward operator $K$, see e.g. \eqref{int:eq111}. The general idea is to rescale the operator $K$ with a weighting operator $S$, such that the smoothness of $\hat{x}$ relative to the rescaled operator $SK$ is better than the original one relative to $K$. In order to avoid distinction of several cases, let us assume for a moment that $q-p>-1$ additional to the assumptions of Theorem \ref{opt1}. Moreover, we assume that $p$ is known to us. The latter is a serious restriction, which will be dropped in the main result Theorem \ref{opt2} below. However, there are settings, where this knowledge is justified, see Example \ref{int:ex1} at the end of this section. For any $\varepsilon>0$ with $p>1+\varepsilon$ we define the (linear and unbounded) weighting operator $S$ as the linear extension of

\begin{align*}
S:\mathcal{D}(S)\subset \mathcal{Y}&\to \mathcal{Y}\\
                                  u_j&\mapsto d_j u_j := j^\frac{p-1-\varepsilon}{2} u_j,\quad j\in\N
\end{align*}
 on
 
  $$\mathcal{D}(S):=\left\{\sum_{j=1}^\infty \alpha_j u_j~:~\sum_{j=1}^\infty \alpha_j^2d_j^2<\infty\right\}.$$
 
Since $q-p>-1$, we directly see that $SK:\mathcal{X}\to\mathcal{Y}$ is compact, with  singular values $d_j \sigma_j \asymp j^{-\frac{q+1+\varepsilon-p}{2}}$ and the same singular bases $(v_j)_{j\in\N}$ and $(u_j)_{j\in\N}$ as $K$. Now assume that $\hat{x}$ obeys smoothness $(\nu,\rho)$ relative to $K$, i.e. there exists $\xi\in\mathcal{Y}$ with $\hat{x}=(K^*K)^\frac{\nu}{2}\xi$ and $\|\xi\|\le \rho$. Let $\nu':=\frac{q}{q+1+\varepsilon-p}\nu$. Note that $\nu'>\nu$, since we assumed that $p>1+\varepsilon$. Then

$$\hat{x}=(K^*K)^\frac{\nu}{2}\xi = \sum_{j=1}^\infty \sigma_j^{\nu}(\xi,v_j)v_j = \sum_{j=1}^\infty (d_j\sigma_j)^{\nu'} \frac{\sigma_j^{\nu-\nu'}}{d_j^{\nu'}}(\xi,v_j)v_j = \left((SK)^*SK\right)^\frac{\nu'}{2}\xi',$$

with $\xi'=\sum_{j=1}^\infty\frac{\sigma_j^{\nu-\nu'}}{d_j}(\xi,v_j)v_j$ and 

$$\|\xi'\|^2 =\sum_{j=1}^\infty \left(\frac{\sigma_j^{\nu-\nu'}}{d_j^{\nu'}}\right)^2(\xi,v_j)^2 \asymp \sum_{j=1}^\infty \frac{j^{-q(\nu-\nu')}}{j^{(p-1-\varepsilon)\nu'}}(\xi,v_j)^2 = \sum_{j=1}^\infty j^{-q\nu\left(1-\frac{q}{q+1+\varepsilon-p}\right)-\frac{p-1-\varepsilon}{q+1+\varepsilon-p}q\nu}(\xi,v_j)^2 = \sum_{j=1}^\infty (\xi,v_j)^2 \le \rho^2,$$

therefore $\hat{x}$ obeys smoothness $(\nu',c \rho)$ relative to $SK$ (with a constant $c>0$). Moreover, the rescaled measurements $SY_1, SY_2,...$ are unbiased estimators of $S\hat{y}$ with finite variance

\begin{equation}\label{int:eq1aa}
\E\|SY_1-S\hat{y}\|^2 = \sum_{j=1}^\infty d_j^2\E(Y_1-\hat{y},u_j)^2 \asymp \sum_{j=1}^\infty j^{p-1-\varepsilon} j^{-p} = \sum_{j=1}^\infty j^{-(1+\varepsilon)}<\infty.
\end{equation}

Our modification of the discrepancy principle is, that we apply it not to the unscaled operator and measurements $K$ and $\bar{Y}_n$, but to the rescaled ones $SK$ and $S\bar{Y}_n$ (note that $SK$ also has dense range). Consequently, the stopping index $k_n$ is the solution of the equation

\begin{equation}\label{int:eq1a}
\sqrt{\sum_{j=k+1}^\infty (S\bar{Y}_n,u_j)^2} = \sqrt{\sum_{j=k+1}^\infty d_j^2(\bar{Y}_n,u_j)^2} \approx {\delta_n^{est}}',
\end{equation}

with

\begin{equation*}
{\delta_n^{est}}':=\frac{1}{\sqrt{n}}\quad\mbox{or}\quad {\delta_n^{est}}':=\frac{\sqrt{\frac{1}{n-1}\sum_{i=1}^n \|SY_i-S\bar{Y}_n\|^2}}{\sqrt{n}}.
\end{equation*}

The following theorem states, that up to $\varepsilon$ the optimal bound from Theorem \ref{opt1} holds with a probability converging to $1$ as $n\to\infty$ using this strategy. Note that convergence in mean squared error cannot be expected for the discrepancy principle, see \cite{harrach2020beyond}. However, adding to the procedure a so-called "emergency stop" (see \cite{blanchard2012discrepancy}, \cite{harrach2020beyond}) might allow to deduce rates in mean squared error, but we will leave this as a future work and focus on the rates in probability.

\begin{theorem}\label{opt1a}
Let $\hat{y}=K\hat{x}$ and assume that $Y_1,Y_2,...$ are i.i.d. for $i=1,2,...$ with $\E[Y_1] = \hat{y}$. Moreover, assume that there are $q>0,p>1$ with $\sigma_j^2 \asymp j^{-q}$ and $\E\left(Y_{1} - \hat{y},u_j\right)^2 \asymp j^{-p}$ and $q>p-1$. Let $\varepsilon>0$ such that $p>1+\varepsilon$ and assume that $\hat{x}\in\mathcal{X}_{\nu,\rho}$. Then there exists $L>0$, such that for $k_n$ the solution of \eqref{int:eq1a} there holds

\begin{equation}\label{int:eq1b}
\mathbb{P}\left(\|\bar{X}^n_{k_n}-\hat{x}\|\le L \rho^\frac{1}{1+\nu'}\left(\frac{1}{\sqrt{n}}\right)^\frac{\nu'}{\nu'+1}\right)\to 1
\end{equation}
 
 as $n\to\infty$.
\end{theorem}

Theorem \ref{opt1a} is an immediate consequence of Theorem 1.2.4 from \cite{jahn2021regularising} (which is a refined version of Theorem 4 of \cite{harrach2020beyond}) applied to $SK$ and $S\bar{Y}_n$. A quick calculation reveals, that $\frac{\nu'}{\nu'+1} = \frac{\nu}{\nu+1-\frac{p-1-\varepsilon}{q}}$, thus up to $\varepsilon>0$ we get the optimal rate from Theorem \ref{opt1}. Note however, that the smaller we choose $\varepsilon$, the slower will be the convergence to $1$ in \eqref{int:eq1b}.

Now we generalise the above result in several ways. We relax the condition $\E(Y_1-\hat{y},u_j)^2\asymp j^{-p}$ to $\E(Y_1-\hat{y},u_j)^2 \le C_p j^{-p}$ (for some $C_p>0$). Most importantly, the exponent $p$ is no longer assumed to be known. Moreover, we account for the fact that in practice we can only measure a finite number of components. The component-wise variances $\E[(Y_1-\hat{y},u_j)^2]$ will be estimated from the multiple measurements and are then used to determine the (now random) rescaling weights $d_{j,n}$. Consequently, the weighting operator $S_n$ is now also random and depends on the samples $Y_1,...,Y_n$. The precise implementation of the modified discrepancy principle (with ad hoc unknown decay of the component-wise variances) is given in Algorithm 1. Hereby, the factor $\sum_{j=1}^{m_n}s_{n,j}^2$ is in essence a normalisation by $\E\|Y_1-\hat{y}\|^2$. The two arguments of the $\min(\cdot)$ function can be roughly interpreted as follows: the first one assures, that the rescaled measurements still have finite variance (c.f. \eqref{int:eq1aa}), while the second one assures, that the rescaled operator is still bounded. We state now the main result, which confirms that the optimal bound from Theorem \ref{opt1} holds with a probability converging to $1$ as $n\to\infty$ (up to discretisation and $\varepsilon_2>0$ arbitrary small in the exponent) using this strategy.

% To demonstrate the basic idea we assume for a moment, that we know the decay of the variances $p$ but the smoothness $\nu$ remains unknown. This would be the typical case under a white noise scenario, if one consideres the (symmetrised) equation $A x = z$ with $A=K^*K$ and $z=K^*y$ for $K$ a Hilbert-Schmidt operator (i.e. for quadratic summable singular values, $q>1$). The baisc idea is to rescale the measurements $Y_1, Y_2,...$ and the operator $K$, such that the smoothness (of $\hat{x}$) relative to the rescaled operator is higher than the one relative to the unscaled. 

%\begin{theorem}
%Let $K\hat{x}=\hat{y}$. Assume that $Y_1, Y_2,...$ are i.i.d. with $\E[Y_1] = \hat{y}$. Moreover, there are $q,p$ with $q>p-1>0$ and  
%\end{theorem}

\begin{theorem}\label{opt2}
Let $K\hat{x}=\hat{y}$ and $0<\varepsilon_1,\varepsilon_2<1$. Assume that $Y_1, Y_2,...$ are i.i.d with $\E[Y_1] = \hat{y}$. Moreover, there are $q,p$ with $q>p-1>0$ and $C_p,C>0$ with $\sigma_j^2 \asymp j^{-q}$ and $\E[\left(Y_{1} - \hat{y},u_j\right)^2]\le C_pj^{-p}$ and $\E[\left(Y_{1}-\hat{y},u_j\right)^4]/\left(\E\left[\left(Y_1-\hat{y},u_j\right)^2\right]\right)^2\le C$ for all $j\in\N$. Assume that $\hat{x}\in\mathcal{X}_{\nu,\rho}$ for $\nu,\rho>0$ and let $k_n$ be the stopping index of the modified discrepancy principle  as implemented in Algorithm 1 with $0<\varepsilon_1<1$ and $0<\varepsilon_2<p-1$ and $Y_1,...,Y_n$. Then there is a $L>0$ such that there holds

\begin{equation}\label{int:eq1ba}
\lim_{n\to\infty}\mathbb{P}\left( \|\bar{X}^n_{k_n}- \hat{x}\| \le L \max\left( \rho^\frac{q+1+\varepsilon_2-p}{(\nu+1)q+1+\varepsilon_2-p} \left(\frac{1}{\sqrt{n}}\right)^\frac{\nu}{\nu+1-\frac{p-1-\varepsilon_2}{q}}, \rho \left(\frac{1}{\sqrt{n}}\right)^{-(1-\varepsilon_1)q\nu}\right)\right)=1.
\end{equation}

\end{theorem} 

The proof of Theorem \ref{opt2} is substantially more difficult than the one of Theorem \ref{opt1a}, mostly due to the dependence of the rescaling operator $S_n$ on the (realisations of) the measurements $Y_1,...,Y_n$. In particular, the $S_nY_1,...,S_nY_n$ are not independent and hence we cannot simply apply the results from \cite{harrach2020beyond}.

 \begin{algorithm}\label{algorithm1}
 \caption{Modified discrepancy principle with estimated data error}
\begin{algorithmic}[1]
\STATE Given measurements $(Y_i,u_j)$ with $i=1,..,n$ and $j=1,...,\lfloor n^{1-\varepsilon_1}\rfloor$;
\STATE \textit{Estimate variances}
\STATE Set $s_{j,n}^2:= \frac{1}{n-1}\sum_{i=1}^n\left( Y_{i} - \bar{Y}_n,u_j\right)^2$;
\STATE \textit{Calculate weights}
\STATE Set $d_{1,n}:=\sqrt{\min\left(\frac{\sum_{j'=1}^{\lfloor n^{1-\varepsilon_1}\rfloor} s_{j',n}^2}{s_{1,n}^2},\frac{1}{\sigma_1^2}\right)}$;
\FOR{ $j=2,...,\lfloor n\rfloor^{1-\varepsilon_1}$}
\STATE Set $d_{j,n}:=\sqrt{\min\left( \frac{j^{-(1+\varepsilon_2)}}{s_{j,n}^2}\sum_{j'=1}^{\lfloor n^{1-\varepsilon_1}\rfloor} s_{j',n}^2, \frac{\sigma_{j-1}^{2}}{\sigma_{j}^2}d_{j-1,n}^2\right) }$;
\ENDFOR
\STATE \textit{Apply discrepancy principle to rescaled measurements}
\STATE Set ${\delta_n^{est}}':= \sqrt{\frac{\sum_{j=1}^{\lfloor n^{1-\varepsilon_1}\rfloor} d_{j,n}^2 s_{j,n}^2 }{n}}$;
\STATE $k=0$
  \WHILE{$\sqrt{\sum_{j=k+1}^{\lfloor n^{1-\varepsilon_1}\rfloor} d_{j,n}^2\left( \bar{Y}_n,u_{j}\right)^2} > {\delta_n^{est}}'$}
  \STATE $k=k+1$;
  \ENDWHILE
  \STATE $k_n=k$;
\end{algorithmic}
\end{algorithm}

\begin{remark}
 Algorithm 1 could be applied in a general setting, e.g. also to severely ill-posed problems. The weights $d_{j,n}$ are defined such that $$d_{1,n}\sigma_1\ge d_{2,n}\sigma_2 \ge ...$$
  Note that a chosen $\varepsilon_2$ fulfills the condition $\varepsilon_2<p-1$, if $\lim_{j\to\infty}\lim_{n\to\infty} d_{j,n} = \infty$ and the latter can be checked to verify, that $\varepsilon_2$ was chosen sufficiently small. The assumption $q>p-1$ is only made for convenience and is not restrictive, since if $\E(Y_1-\hat{y},u_j)^2\le C_pj^{-p}$ there also holds that $\E(Y_1-\hat{y},u_j)^2\le C_p j^{-p'}$ for all $p'\le p$. 
\end{remark}

The second argument of the maximum in \eqref{int:eq1ba} is a discretisation error due to the usage of only finitely many singular vectors. If the latter is negligible, i.e. if $\frac{\nu}{\nu+1}<(1-\varepsilon)q\nu$, the rate from Theorem \ref{opt2} is better than the deterministic worst-case rate from \cite{harrach2020beyond}.
The additional assumption $\sup_{j\in\N} \frac{\E[\left(Y_1-\hat{y},u_j\right)^4]}{\left(\E[(Y_1-\hat{y},u_j)^2]\right)^2}<\infty$ assures that the component distribution are not too degenerated.  This is clearly fulfilled, if $\E(Y_1-\hat{y},u_j) \stackrel{d}{=} c_jZ$ for some $Z$ with $\E[Z]=0$, $E[Z^4]<\infty$ and $(c_j)_{j\in\N}\subset \R \setminus\{0\}$ (e.g. this holds under Gaussian noise). In particular no independence between the components is required. Our assumption of finite variance ($\E\|Y_1-\hat{y}\|^2<\infty$) excludes a direct application to white noise scenarios. The following example shows, how to adapt the approach for Hilbert-Schmidt operators under white noise.

\begin{example}\label{int:ex1}

Consider the equation $A\hat{x}=\hat{z}$ for $A:\mathcal{X}\to\mathcal{Y}$ Hilbert-Schmidt and assume there is a $q>1$ such that $\sigma_j(A)^2 \asymp j^{-q}$. Assume that the measurements are corrupted by i.i.d centered Hilbert-space processes $Z_1, Z_2,...$ (operating on $\mathcal{Y}$). I.e., the $Z_i:\mathcal{Y}\to L^2(\Omega,\mathcal{A},\mathbb{P})$ are bounded linear operators from $\mathcal{Y}$ to the space of square-integrable real-valued random variables (on some probability space $(\Omega,\mathcal{A},\mathcal{P})$), such that  $\E(Z_1,z)=0$. Moreover, $Z_i$ has an arbitrary covariance operator $\mathbf{Cov}_Z:\mathcal{Y}\to\mathcal{Y}$, which is the bounded linear operator defined implicitly via the equation  $(\mathbf{Cov}_Z z,z')=\E\left[(Z_1,z)(Z_1,z')\right]$ for all $z,z'\in\mathcal{Y}$ (the case where $\mathbf{Cov}_Z=Id$ is denoted as white noise).
 Instead of $Ax=z$ we solve the symmetrised equation $K\hat{x}=\hat{y}$, with $K=A^*A$ and $y=A^*\hat{z}$. The symmetrised i.i.d. measurements $Y_1=A^*(\hat{z}+Z_1), Y_2 = A^*(\hat{z}+Z_2),...$ then fulfill $\E[Y_1] = A^*\hat{z} = \hat{y}$ and 
 
 \begin{align*}
 \E(Y_1-\hat{y},u_j(K))^2 &= \E(A^*(\hat{z}+Z_1) - A^*\hat{z},v_j(A))^2 = \sigma_j(A)^2\E(Z_1,u_j(A))^2\\
  &\asymp j^{-q} (\mathbf{Cov}_Z u_j(A),u_j(A))^2 \le j^{-q} \|\mathbf{Cov}_Z\|^2 = j^{-q}.
 \end{align*}
 
Since we assume to know the singular value decomposition, this allows to apply Theorem \ref{opt1a}. It should be noted, that for coloured noise, no prewhitening step or additional assumptions for the covariance operator are needed (in contrast to e.g. \cite{blanchard2012discrepancy}).

\end{example}

 All in all, the main contribution of this work is to answer the question of optimal adaptivity (in the minimax-sense) for the discrepancy principle in statistical inverse problems with multiple measurements, which was left open in the original work \cite{harrach2020beyond}. Moreover, in the light of Example \ref{int:ex1} the results may be compared to classical existing results for statistical inverse problems, usually using a white noise error model. In particular, they generalise results from \cite{blanchard2012discrepancy}, \cite{jin2013oracle} and \cite{lu2014discrepancy}, where modifications of the discrepancy principle are applied to the symmetrised equation in several ways. Firstly, the error distribution is arbitrary, secondly the noise level and the covariance structure need not to be known and thirdly, a self-similarity condition (Assumption 3 in \cite{lu2014discrepancy} and Assumption 2.4 in \cite{jin2013oracle}) for $\hat{x}$ is not needed. However it should be mentioned here that using symmetrisation is usually avoided, since the ill-posedness of the symmetrised equation $A^*Ax=A^*y$ is much worse than the one of the original equation $Ax=y$. Note that this does not contradict the (almost) order-optimality of the methods relying on the discrepancy principle mentioned above, but still may cause problems in practice. Because of this, under white noise one often relies on other methods, which do not depend on the residual, see e.g. \cite{bissantz2007convergence} for a priori bounds, \cite{mathe2003discretization} for the Lepski principle, or \cite{li2020empirical} for unbiased risk estimation, to only name a few. Finally in \cite{blanchard2018early}, \cite{harrach2020regularising} and \cite{jahn2021optimal} recent modifications of the discrepancy principle, which are based on discretisation and not on symmetrisation, are investigated in white noise scenarios.

\section{Proofs}\label{sec:3}

In this section we present the proofs of the above statements.

\subsection{Proof of Theorem \ref{opt1}}

 Note that $(\hat{y},u_j) = \sigma_j (\hat{x},v_j) = \sigma_j^{1+\nu} (\xi,v_j)$. The bias-variance decomposition gives

\begin{align*}
\E\|\bar{X}_k^n-\hat{x}\|^2 &=\E\left[\sum_{j=1}^k\left(\frac{(\bar{Y}_n,u_j)}{\sigma_j} - (\hat{x},v_j)\right)^2\right] + \sum_{j=k+1}^\infty(\hat{x},v_j)^2\\
&= \sum_{j=1}^k \sigma_j^{-2}\E(\bar{Y}_n-\hat{y},u_j)^2 + \sum_{j=k+1}^\infty (\hat{x},v_j)^2\\
 & = \frac{1}{n} \sum_{j=1}^k \sigma_j^{-2} \E(Y_1-\hat{y},u_j)^2 + \sum_{j=k+1}^\infty \sigma_j^{2\nu}(\xi,v_j)^2\\
 & \asymp \frac{1}{n} \sum_{j=1}^k j^{q-p} + \sum_{j=k+1}^\infty j^{-\nu q}(\xi,v_j)^2.
\end{align*}

 Therefore it holds that

\begin{align*}
\sup_{\hat{x}\in\mathcal{X}_{\nu,\rho}}\E\|\bar{X}_k^n - \hat{x}\|^2 &\asymp \frac{1}{n} \int_{j=1}^k x^{q-p} dx + \rho k^{-\nu q}\\
&\asymp \begin{cases} \frac{1}{n} + \rho k^{-\nu q} & q-p<-1\\
\frac{1}{n}\log(k) + \rho k^{-\nu q} & q-p=-1\\
\frac{1}{n}k^{p-q+1} +\rho k^{-\nu q} & q-p>-1
\end{cases}.
\end{align*}

The right hand side is minimised by the choices

$$k=k_n \asymp \begin{cases} (\rho n)^{\frac{1}{\nu q}} & q-p\le-1\\
                              (\rho n)^\frac{1}{(1+\nu)q + 1 -p} & q-p>-1
                              \end{cases}.$$

Thus we obtain 

$$\inf_{k\ge 1}\sup_{\hat{x}\in\mathcal{X}_{\nu,\rho}}\E\| \bar{X}_{k}^{n} - \hat{x}\|^2\asymp\sup_{\hat{x}\in\mathcal{X}_{\nu,\rho}}\E\| \bar{X}_{k_n}^{n} - \hat{x}\|^2 \asymp \begin{cases} \frac{1}{n} & q-p<-1\\
                     \frac{\log(\rho n)}{n} & q-p=-1\\
                     \rho^\frac{q+1-p}{(\nu+1)q+1-p}\left(\frac{1}{n}\right)^\frac{\nu}{\nu+1 - \frac{p-1}{q}} & q-p>-1
                     \end{cases}.$$

\subsection{Proof of Theorem \ref{opt2}}

 Let $m_n:=\lfloor n^{1-\varepsilon_1}\rfloor$ and $\hat{x}=(K^*K)^\frac{\nu}{2} \xi$ with $\|\xi\|\le \rho$. For $j\in\N$ fixed  it holds that $s_{j,n}^2\to \E(Y_1-\hat{y},u_j)^2$ (in probability, almost surely and in $L^2$). We denote the deterministic limit (for $n\to\infty$) of the random weights $d_{j,n}$ by 

\begin{align*}
d_1:&=\sqrt{\min\left(\frac{\E\|Y_1-\hat{y}\|^2}{\E\left(Y_{1}-\hat{y},u_1\right)^2},\frac{1}{\sigma_1^2}\right)},\\
d_{j}:&=\sqrt{\min\left(\frac{j^{-(1+\varepsilon_2)}}{\E\left(Y_{1}-\hat{y},u_j\right)^2}\E\|Y_1-\hat{y}\|^2,\frac{\sigma_{j-1}^2}{\sigma_j^2} d_{j-1}^2\right)},\quad j>1.
\end{align*}

The weights $d_{j,n}$ and $d_j$ can be interpreted as belonging to weighting operators $S_n$ and $S$ respectively. Moreover, the assumption on the error distribution of the $(Y_1-\hat{y},u_j)$ imply that $S$ can be seen as a deterministic limit (for $n\to\infty$) of the $S_n$ in a suitable sense. This will ultimately allow to rephrase the increased smoothness relative to the deterministic rescaled limit operator $SK$ instead of the random rescaled operator $S_nK$. In the following, $S_n$ and $S$ are not used explicitly, it suffices to stick to the weights $d_{j,n}$ and $d_j$. We start with the following auxiliary proposition, which summarises some of the properties of the sequence $(d_j)_{j\in\N}$.
\begin{proposition}\label{sec3:prop1}
There holds
\begin{align}\label{opt2:eq0}
d_j&\le \frac{1}{\sigma_j},\\\label{opt2:eq0b}
\lim_{j\to\infty} d_j&=\infty,\\\label{opt2:eq0c}
\inf_{j\in\N} d_j&=:d>0.
\end{align}
\end{proposition}

\begin{proof}[Proof of Proposition \ref{sec3:prop1}]

Note that obviously $d_j>0$ for all $j\in\N$. First, \eqref{opt2:eq0} is fulfilled for $j=1$. For $j\ge 2$, we have

\begin{equation}
d_j\le \frac{\sigma_{j-1}}{\sigma_j}d_{j-1} \le \frac{\sigma_{j-1}}{\sigma_j}\frac{\sigma_{j-2}}{\sigma_{j-1}}d_{j-2} \le \frac{\sigma_{j-1}}{\sigma_j}\frac{\sigma_{j-2}}{\sigma_{j-1}} ... \frac{\sigma_1}{\sigma_2}d_1\le  \frac{\sigma_{j-1}}{\sigma_j}\frac{\sigma_{j-2}}{\sigma_{j-1}} ... \frac{\sigma_1}{\sigma_2}\frac{1}{\sigma_1} = \frac{1}{\sigma_j}.
\end{equation}

For \eqref{opt2:eq0b} set $J:=\sup\left\{j\in\N~:~d_j=\sqrt{\frac{j^{-(1+\varepsilon_2)}}{\E(Y_1-\hat{y},u_j)^2}\E\|Y_1-\hat{y}\|^2}\right\}$. If $J=\infty$, the statement is proven since from $p>1+\varepsilon_2$ it follows that

$$\frac{j^{-(1+\varepsilon_2)}}{\E(Y_1-\hat{y},u_j)^2} \ge j^{p-(1+\varepsilon_2)} C_p^{-1} \to \infty$$

as $j\to\infty$. Otherwise, if $J<\infty$, there holds $d_j=\frac{\sigma_{j-1}}{\sigma_j}d_{j-1}$ for $j\ge J$ and thus

$$d_j = \frac{\sigma_{j-1}}{\sigma_j}d_{j-1} = \frac{\sigma_{j-1}}{\sigma_j}  \frac{\sigma_{j-2}}{\sigma_{j-1}} ... \frac{\sigma_{J}}{\sigma_{J+1}} d_J = \frac{\sigma_J}{\sigma_j} d_J \to \infty$$

as $j\to\infty$, since $\sigma_j\to 0$. Finally, \eqref{opt2:eq0c} follows directly from \eqref{opt2:eq0b}.

\end{proof}

  Now we first show, that the true solution has at least smoothness $\nu':=\frac{q}{q+1+\varepsilon_2-p} \nu$ (relative to the rescaled limit operator $SK$). Since $\varepsilon_2<p-1$, there holds $\nu'>\nu$.
We use (the reverse of) \eqref{opt2:eq0} together with \eqref{opt2:eq0c} and obtain

\begin{align}\label{opt2:eq0a}
\frac{\sigma_j^{\nu-\nu'}}{d_j^{\nu'}} &\le \frac{d_j^{\nu'-\nu}}{d_j^{\nu'}} = d_j^{-\nu} \le d^{-\nu}
\end{align}

for all $j\in\N$. We express $\hat{x}$ with respect to the rescaled limit operator $SK$ and obtain

\begin{align}\label{opt2:eq1}
\hat{x} & =\sum_{j=1}^\infty \sigma_j^\nu (\xi,v_j)v_j = \sum_{j=1}^\infty \left(d_j \sigma_j\right)^{\nu'} \frac{\sigma_j^{\nu-\nu'}}{d_j^{\nu'}}(\xi,v_j) v_ j  = \sum_{j=1}^\infty (d_j\sigma_j)^{\nu'}(\xi',v_j)v_j
\end{align}

with $\xi':=\sum_{j=1}^\infty \frac{\sigma_j^{\nu-\nu'}}{d_j^{\nu'}}(\xi,v_j)v_j$. By \eqref{opt2:eq0a}, there holds $\|\xi'\| \le d^{-\nu}\|\xi\| \le d^{-\nu}\rho =:\rho'$.  

The assumption $\sup_{j\in\N}\frac{\E\left[\left(Y_1-\hat{y},u_j\right)^4\right]}{\left(\E\left[\left(Y_1-\hat{y},u_j\right)^2\right]\right)^2}$ guarantees, that we can estimate the variances $\E(Y_1-\hat{y},u_j)^2$ uniformly for $j=1,...,m_n$. To see this, we use Theorem 2 of \cite{angelova2012moments} which states that $\E[|s_{j,n}^2-\E(Y_1-\hat{y},u_j)^2|^2] \le \frac{4}{n} \E(Y_1-\hat{y},u_j)^4$. Let $\eta>0$. Then

\begin{align*}
&\mathbb{P}\left(|s_{j,n}^2-\E(Y_1-\hat{y},u_j)^2| \le \eta\E(Y_1-\hat{y},u_j)^2,~ \forall j\le m_n\right)\\
\ge &1 - \sum_{j=1}^{m_n}\mathbb{P}\left(|s_{j,n}^2-\E(Y_1-\hat{y},u_j)^2| > \eta\E(Y_1-\hat{y},u_j)^2\right)\\
\ge &1 - \sum_{j=1}^{m_n}\frac{\E[|s_{j,n}^2-\E(Y_1-\hat{y},u_j)^2|^2]}{(\eta\E[(Y_1-\hat{y},u_j)^2])^2}\\
 = &1- \frac{4m_n}{\eta^2n}\sup_{j=1,...,m_n}\frac{\E[(Y_1-\hat{y},u_j)^4]}{(\E[(Y_1-\hat{y},u_j)^2])^2} \ge 1 -\frac{4m_nC}{\eta^2n}\\
  \ge &1- \frac{4}{\eta^2}C n^{-\varepsilon_1} \to 1
\end{align*}

as $n\to\infty$, where we used Chebyshev's inequality in the second step. From that directly follows

\begin{align}\label{thrdp:eq2}
\mathbb{P}&\left( \frac{d_j}{2}\le d_{j,n} \le 2 d_{j},~\forall j=1,...,m_n\right)\to 1,\\\label{thrdp:eq1}
\mathbb{P}&\left( |\sqrt{n}\delta_n^{est} - \gamma|\le \frac{\gamma}{2}\right)\to 1
\end{align}

for 

\begin{equation}\label{thrdp:eq2a}
{\delta_n^{est}}' := \sqrt{\frac{\sum_{j=1}^{m_n} d_{j,n}^2 s_{j,n}^2}{n}}
\end{equation}

from Algorithm 1 and $\gamma:=\sqrt{\sum_{j=1}^{\infty} d_j^2\E(Y_1-\hat{y},u_j)^2}$ as $n\to\infty$. 

We will distinguish two cases in the following. In the analysis we will often restrict to certain good events $\Omega_n$ which hold with a probability $\mathbb{P}\left(\Omega_n\right)\to 1$ as $n\to\infty$. Moreover, we will repeatedly use Markov/Chebyshev's inequality and that for i.i.d real-valued random variables $Z_1,...,Z_n$ with $\E[Z_i]=0$ and $\E[Z_i^2]<\infty$ there holds $\E[\left(\sum_{i=1}^nZ_i\right)^2] = n\E[Z_1^2]$. Thus, e.g. 

$$\E\left[(\bar{Y}_n-\hat{y},u_j)^2\right] = \frac{1}{n}\E[(Y_1-\hat{y},u_j)^2]$$

for all $j,n\in\N$.
\subsubsection{Case 1}

We first assume, that for all $k\in\N$ there exists $j_k\ge k$ such that $(\hat{y},u_{j_k})\neq 0$. Note that then also $(\hat{x},v_{j_k}), (\xi,v_{j_k})\neq 0$.

\begin{lemma}\label{chapt1:lem2}
Assume that for all $k\in\N$ there exists $j_k\ge k$ such that $(\hat{y},u_{j_k})\neq 0$. Then for ${\delta_n^{est}}'$ from \eqref{thrdp:eq2a} there holds 

$$\mathbb{P}\left( \sqrt{\sum_{j=k_n}^{m_n} d_{j,n}^2 (\bar{Y}_n-\hat{y},u_j)^2} \le \frac{{\delta_n^{est}}'}{2}\right)\to 1$$

as $n\to\infty$.

\end{lemma}

\begin{proof}[Proof of Lemma \ref{chapt1:lem2}]

We first show that there exists $(q_n)_{n\in\N}$ such that

\begin{equation}\label{thrdp:eq3}
\mathbb{P}\left(k_n\ge q_n\right)\to1\quad\mbox{and}\quad q_n\to \infty
\end{equation}

as $n\to\infty$. For that it suffices to show that $\lim_{n\to\infty}\mathbb{P}\left( k_n\ge k\right) =1$ for all $k\in\N$. By assumption there exists $j_k\ge k$ such that $(\hat{y},u_{j_k})\neq 0$. We set

\begin{equation}\label{thrdp:eq5}
\Omega_n:=\left\{ |\sqrt{n}{\delta_n^{est}}' - \gamma| \le \frac{\gamma}{2},~(\bar{Y}_n,u_{j_k})^2\ge (\hat{y},u_{j_k})^2/2,~\frac{d_{j_k,n}}{2}\le d_{j_k} \le 2d_{j_k,n}\right\}.
\end{equation}

Then for $n\ge \max(j_k,32\gamma^2/(d_{j_k}(\hat{y},u_{j_k}))^2)$,

\begin{align*}
{\delta_n^{est}}'\chi_{\Omega_n} &\le \frac{2\gamma}{\sqrt{n}}\chi_{\Omega_n} < \sqrt{\frac{d_{j_k}^2(\hat{y},u_{j_k})^2}{8}}\chi_{\Omega_n} \le \sqrt{\frac{d_{j_k}^2(\bar{Y}_n,u_{j_k})^2}{2}}\chi_{\Omega_n}\\
&\le \sqrt{d_{j_k,n}^2(\bar{Y}_n,u_{j_k})^2}\le \sqrt{\sum_{j=k+1}^{m_n}d_{j,n}^2(\bar{Y}_n,u_j)^2}.
\end{align*}

Thus $k_n\chi_{\Omega_n} \ge k\chi_{\Omega_n}$  by Algorithm 1 and \eqref{thrdp:eq3} follows with $\lim_{n\to\infty}\mathbb{P}\left(\Omega_n\right)=1$, which holds because of \eqref{thrdp:eq2}, \eqref{thrdp:eq1} and the law of large numbers. We come to the main proof. Let $\varepsilon>0$ and $q_n$ be such that $q_n\to\infty$ and

\begin{equation}\label{thrdp:eq5a}
\mathbb{P}\left(k_n\ge q_n\right)\to1
\end{equation}
  as $n\to\infty$. Then,

\begin{align}\label{thrdp:eq5b}
&\mathbb{P}\left( \sqrt{\sum_{j=k_n}^{m_n}d_{j,n}^2(\bar{Y}_n-\hat{y},u_j)^2} \le \frac{{\delta_n^{est}}'}{2}\right)\\\notag
\ge &\mathbb{P}\left( \sqrt{\sum_{k=q_n}^{m_n}d_{j}^2(\bar{Y}_n-\hat{y},u_j)^2} \le \frac{\gamma}{8\sqrt{n}},~{\delta_n^{est}}'\ge \frac{\gamma}{2\sqrt{n}},~ 2d_j\ge d_{j,n}~\forall j\le m_n,~k_n\ge q_n\right)\\\notag
\ge &1-\mathbb{P}\left( \sqrt{\sum_{k=q_n}^{m_n}d_{j}^2(\bar{Y}_n-\hat{y},u_j)^2} >\frac{\gamma}{8\sqrt{n}}\right) - \mathbb{P}\left({\delta_n^{est}}'< \frac{\gamma}{2\sqrt{n}}\right)\\\notag
&\qquad-\mathbb{P}\left(\exists j\le m_n \mbox{~such that~} d_{j,n}>2d_j\right) - \mathbb{P}\left(k_n< q_n\right).
\end{align}

 Now

\begin{align}\label{thrdp:eq5c}
\mathbb{P}\left( \sqrt{\sum_{j=q_n}^{m_n} d_j^2(\bar{Y}_n-\hat{y},u_j)^2} > \frac{\gamma}{8\sqrt{n}}\right) &\le \frac{64n}{\gamma^2} \sum_{j=q_n}^{m_n} d_j^2\E(\bar{Y}_n-\hat{y},u_j)^2 =  \frac{64 }{\gamma^2} \sum_{j=q_n}^{m_n} d_j^2\E(Y_1-\hat{y},u_j)^2\\\notag
&= \frac{64}{\gamma^2} \sum_{j=q_n}^{m_n} \min\left(\frac{j^{-(1+\varepsilon_2)}}{\E(Y_1-\hat{y},u_j)^2},\sigma_j^{-2}\right)\E(Y_1-\hat{y},u_j)^2\\\notag
 &= \frac{64}{\gamma^2}\sum_{j=q_n}^{m_n} j^{-(1+\varepsilon_2)}\le \frac{64}{\gamma^2} \sum_{j=q_n}^{m_n} j^{-(1+\varepsilon_2)} \to 0
\end{align}

as $n\to\infty$, where we used $\sum_{j=1}^\infty j^{-(1+\varepsilon_2)} <\infty$ and $q_n\to \infty$ in the sixth step. Plugging \eqref{thrdp:eq2}, \eqref{thrdp:eq1}, \eqref{thrdp:eq5a} and \eqref{thrdp:eq5c} into \eqref{thrdp:eq5b} then yields

$$\mathbb{P}\left(\sqrt{\sum_{k=q_n}^{m_n} d_{j,n}^2 (\bar{Y}_n-\hat{y},u_j)^2} \le \frac{{\delta_n^{est}}'}{2}\right)\to 1$$

as $n\to\infty$ and the proof of Lemma \ref{chapt1:lem2} is concluded.

\end{proof}

We start the main proof and decompose as usual into a data propagation error, approximation error and discretisation error

$$\|\bar{X}^n_{k_n}-\hat{x}\| \le \sqrt{\sum_{j=1}^{k_n} \frac{(\bar{Y}_n-\hat{y},u_j)^2}{\sigma_j^2}} + \sqrt{\sum_{j=k_n+1}^{m_n}(\hat{x},v_j)^2} + \sqrt{\sum_{j=m_n+1}^\infty (\hat{x},v_j)^2}.$$

We first consider the second term (approximation error). With the convention $\sum_{j=s}^t = 0$ for $s>t$, a standard application of H\"older's inequality for $p=\frac{\nu'+1}{\nu'}$ and $q=\nu'+1$, \eqref{opt2:eq1} and the triangle inequality yield

\begin{align*}
&\sqrt{\sum_{j=k_n+1}^{m_n} (\hat{x},u_j)^2} =\sqrt{\sum_{j=k_n+1}^{m_n} (d_j \sigma_j)^{2\nu'}(\xi',v_j)^2}\\
 \le &\sqrt{ \left(\sum_{j=k_n+1}^{m_n}(d_j \sigma_j)^{2(\nu'+1)} (\xi',v_j)^2\right)^\frac{\nu'}{\nu'+1}\left(\sum_{j=k_n+1}^{m_n} (\xi',v_j)^2\right)^\frac{1}{\nu'+1}}\\
 \le &{\rho'}^\frac{1}{\nu'+1}\left(\sqrt{ \sum_{j=k_n+1}^{m_n} (d_j\sigma_j)^2(\hat{x},v_j)^2}\right)^\frac{\nu'}{\nu'+1}={\rho'}^\frac{1}{\nu'+1}\left(\sqrt{ \sum_{j=k_n+1}^{m_n} d_j^2 (\hat{y},v_j)^2}\right)^\frac{\nu'}{\nu'+1}\\
 \le &{\rho'}^\frac{1}{\nu'+1}\left( \sqrt{\sum_{j=k_n+1}^{m_n} d_j^2(\bar{Y}_n,u_j)^2} + \sqrt{\sum_{j=k_n+1}^{m_n} d_j^2(\bar{Y}_n-\hat{y},u_j)^2}\right)^\frac{\nu'}{\nu'+1}.
\end{align*}

Thus for 

\begin{align}\label{thrdp:eq6}
\Omega_n:&=\left\{\sqrt{\sum_{j=k_n}^{m_n} d_{j,n}^2(\bar{Y}_n-\hat{y},u_j)^2} \le \frac{{\delta_n^{est}}'}{2},~|{\delta_n^{est}}' - \frac{\gamma}{\sqrt{n}}|\le \frac{\gamma}{2\sqrt{n}},\right.\\\notag
&\qquad \left.~\frac{d_{j}}{2}\le d_{j,n}\le 2 d_j~\forall j\le m_n\right\}
\end{align}

there holds

\begin{align*}
&\left( \sqrt{\sum_{j=k_n+1}^{m_n} d_j^2(\bar{Y}_n,u_j)^2} + \sqrt{\sum_{j=k_n+1}^{m_n} d_j^2(\bar{Y}_n-\hat{y},u_j)^2}\right)^\frac{\nu'}{\nu'+1}\chi_{\Omega_n}\\
\le &2^\frac{\nu'}{\nu'+1}\left( \sqrt{\sum_{j=k_n+1}^{m_n} d_{j,n}^2(\bar{Y}_n,u_j)^2} + \sqrt{\sum_{j=k_n+1}^{m_n} d_{j,n}^2(\bar{Y}_n-\hat{y},u_j)^2}\right)^\frac{\nu'}{\nu'+1}\chi_{\Omega_n}\\
\le &2^\frac{\nu'}{\nu'+1} \left( {\delta_n^{est}}'+\frac{{\delta_n^{est}}'}{2}\right)^\frac{\nu'}{\nu'+1}\chi_{\Omega_n} \le \left(\frac{4 \gamma}{\sqrt{n}}\right)^\frac{\nu'}{\nu'+1}
\end{align*}

by the Definition of $\Omega_n$ \eqref{thrdp:eq6} and $k_n$.
Consequently, for the approximation error and the discretisation error there holds

\begin{align}\label{thrdp:eq7}
&\sqrt{\sum_{j=k_n+1}^{m_n} (\hat{x},u_j)^2}\chi_{\Omega_n} + \sqrt{\sum_{j=m_n+1}^\infty (\hat{x},u_j)^2}\\\notag
 \le &{\rho'}^\frac{1}{\nu'+1}\left(\frac{4\gamma}{\sqrt{n}}\right)^\frac{\nu'}{\nu'+1} + \sqrt{\sum_{j=m_n+1}^\infty \sigma_j^{2\nu}(\xi,v_j)^2}\le \left(d^{-\nu}\rho\right)^\frac{1}{\nu'+1}\left(\frac{4\gamma}{\sqrt{n}}\right)^\frac{\nu'}{\nu'+1} + \sigma_{m_n}^{\nu} \rho\\\notag
 \le &\frac{L}{2} \max\left({\rho}^\frac{q+1+\varepsilon_2-p}{(\nu+1)q+1+\varepsilon_2-p} \left(\frac{1}{\sqrt{n}}\right)^\frac{\nu}{\nu+\frac{q+1+\varepsilon_2-p}{q}}, \rho \left(\frac{1}{\sqrt{n}}\right)^{(1-\varepsilon_1)q\nu}\right)
\end{align}

for $L=4\max\left((4\gamma)^\frac{\nu'}{\nu'+1} d^{-\frac{\nu}{\nu+1}}, d^{-\frac{\nu}{\nu+1}},1\right)$ and we obtain

\begin{align}\label{case1approx}
&\mathbb{P}\left(\sqrt{\sum_{j=k_n+1}^{m_n}(\hat{x},v_j)^2} + \sqrt{\sum_{j=m_n+1}^\infty (\hat{x},v_j)^2} \le \frac{L}{2}\max\left(\rho^\frac{q+1+\varepsilon_2-p}{(\nu+1)q+1+\varepsilon_2-p}\left(\frac{1}{\sqrt{n}}\right)^\frac{\nu}{\nu+\frac{q+1+\varepsilon_2-p}{q}},\rho\left(\frac{1}{\sqrt{n}}\right)^{(1-\varepsilon_1)q\nu}\right)\right)\\\notag
\ge &\mathbb{P}\left(\Omega_n\right)\to 1
\end{align}

as $n\to\infty$, where we used \eqref{thrdp:eq2}, \eqref{thrdp:eq1} and Lemma \ref{chapt1:lem2} for $\Omega_n$ given in \eqref{thrdp:eq6}.

 To finish the proof we need to verify a similar bound for the data propagation error. By definition of the discrepancy principle (Algorithm 1) and $\Omega_{n}$ in \eqref{thrdp:eq6} there holds

\begin{align*}
{\delta_n^{est}}'\chi_{\Omega_n} &< \sqrt{\sum_{j=k_n}^{m_n} d_{j,n}^2(\bar{Y}_n,u_j)^2}\chi_{\Omega_n}\\
 &\le \sqrt{\sum_{j=k_n}^{m_n}d_{j,n}^2(\hat{y},u_j)^2}\chi_{\Omega_n}  + \sqrt{\sum_{j=k_n}^{m_n}d_{j,n}^2(\bar{Y}_n-\hat{y},u_j)^2}\chi_{\Omega_n}\\
&< 2\sqrt{\sum_{j=k_n}^{m_n} d_j^2(\hat{y},u_j)^2}'\chi_{\Omega_n} + \frac{{\delta_n^{est}}'}{2}\chi_{\Omega_n} = 2 \sqrt{\sum_{j=k_n}^{m_n} (d_j\sigma_j)^{2(1+\nu')} (\xi',v_j)^2}\chi_{\Omega_n}+\frac{{\delta_n^{est}}'}{2}\chi_{\Omega_n} \\
&\le 2 {\rho'} (d_{k_n}\sigma_{k_n})^{\nu'+1}\chi_{\Omega_n} + \frac{{\delta_n^{est}}'}{2}\chi_{\Omega_n} \\
\Longrightarrow \frac{1}{d_{k_n}\sigma_{k_n}}\chi_{\Omega_n} &< \left(\frac{4\rho'}{{\delta_n^{est}}'}\right)^\frac{1}{\nu'+1}\chi_{\Omega_n} \le\left(\frac{16 d^{-\nu} \rho \sqrt{n}}{\gamma}\right)^\frac{1}{\nu'+1}\chi_{\Omega_n}=:b_n\chi_{\Omega_n},
\end{align*}

where we used that $d_1\sigma_1\ge d_2\sigma_2\ge ...$ by definition of $d_{j}$. So,

\begin{equation}\label{thrdp:eq7a}
\mathbb{P}\left( d_{k_n}\sigma_{k_n} > b_n^{-1}\right)\ge \mathbb{P}\left(\Omega_n\right)\to 1
\end{equation}

as $n\to\infty$, for $\Omega_n$ given in \eqref{thrdp:eq6}. Now we show, that for all $\varepsilon>0$ it holds that

\begin{equation}\label{thrdp:eq7a1} 
 \mathbb{P}\left(\sqrt{\sum_{j=1}^{k_n} \frac{(\bar{Y}_n-\hat{y},u_j)^2}{\sigma_j^2}}\le \frac{L}{2} \rho^\frac{1}{\nu'+1}\left(\frac{1}{\sqrt{n}}\right)^\frac{\nu'}{\nu'+1}\right) \ge 1-3\varepsilon 
\end{equation}

for $n$ large enough. Let $j_\varepsilon$ be such that $\sum_{j\ge j_\varepsilon}j^{-(1+\varepsilon_2)} \le \varepsilon\left(\frac{\gamma}{16d^{-\nu}}\right)^2$ and set $C_\varepsilon:= \sqrt{\sum_{j=1}^{j_\varepsilon} \frac{\E(Y_1-\hat{y},u_j)^2}{\sigma_j^2}}$. Define

\begin{align}\label{thrdp:eq7a2}
\Omega_{n,\varepsilon}=\left\{d_{k_n}\sigma_{k_n}\ge b_n^{-1},~\sqrt{\sum_{j=1}^{j_\varepsilon}\frac{(\bar{Y}_n-\hat{y},u_j)^2}{\sigma_j^2}}\le C_\varepsilon \sqrt{n}^{\varepsilon'-1},~\sqrt{\sum_{j=j_\varepsilon+1}^\infty d_j^2(\bar{Y}_n-\hat{y},u_j)^2}\le \left(\frac{\gamma}{16d^{-\nu}}\right)^\frac{1}{\nu'+1} \sqrt{n}^{-1}\right\}
\end{align}

 with $\varepsilon'=\frac{1}{2}\frac{1}{\nu'+1}$. Then,

\begin{align*}
\sqrt{\sum_{j=1}^{k_n}\frac{(\bar{Y}_n-\hat{y},u_j)^2}{\sigma_j^2}}\chi_{\Omega_{n,\varepsilon}}&\le \sqrt{\sum_{j=1}^{j_\varepsilon} \frac{(\bar{Y}_n-\hat{y},u_j)^2}{\sigma_j^2}}\chi_{\Omega_{n,\varepsilon}} + \sqrt{\sum_{j=j_\varepsilon+1}^{k_n} \frac{d_j^2(\bar{Y}_n-\hat{y},u_j)^2}{d_j^2\sigma_j^2}}\chi_{\Omega_{n,\varepsilon}}\\
&\le C_\varepsilon \sqrt{n}^{\varepsilon'-1} + \frac{1}{d_{k_n}\sigma_{k_n}} \sqrt{\sum_{j=j_\varepsilon+1}^{k_n} d_j^2(\bar{Y}_n-\hat{y},u_j)^2}\chi_{\Omega_{n,\varepsilon}}\\
&\le C_\varepsilon \sqrt{n}^{\varepsilon'-1} + b_n\chi_{\Omega_{n,\varepsilon}} \left(\frac{\gamma}{16C_2}\right)^\frac{1}{\nu'+1} \sqrt{n}^{-1}\\
&\le C_\varepsilon \sqrt{n}^{\varepsilon'-1} + \rho^\frac{1}{\nu'+1}\sqrt{n}^{\frac{1}{\nu'+1}-1}\le \frac{L}{2}\rho^\frac{1}{\nu'+1}\left(\frac{1}{\sqrt{n}}\right)^\frac{\nu'}{\nu'+1}
\end{align*}

for $n$ large enough, since $\varepsilon'<\frac{1}{\nu'+1}$. To prove \eqref{thrdp:eq7a1} it remains to show that $\mathbb{P}\left(\Omega_{n,\varepsilon}\right)\ge1-3\varepsilon$ for $n$ large enough. We apply Markov's inequality and obtain

\begin{align}\label{thrdp:eq7b}
\mathbb{P}\left(\sqrt{\sum_{j=1}^{j_\varepsilon}\frac{\left(\bar{Y}_n-\hat{y},u_j\right)^2}{\sigma_j^2}} > C_\varepsilon \sqrt{n}^{\varepsilon'-1}\right)&\le C_\varepsilon^{-2}n^{1-\varepsilon'}\sum_{j=1}^{j_\varepsilon}\frac{\E(\bar{Y}_n-\hat{y},u_j)^2}{\sigma_j^2} = C_\varepsilon^{-2}n^{-\varepsilon'}\sum_{j=1}^{j_\varepsilon}\frac{\E(Y_1-\hat{y},u_j)^2}{\sigma_j^2}=n^{-\varepsilon'}\le \varepsilon
\end{align}

for $n$ large enough by definition of $C_\varepsilon$. Further, by the choice of $j_\varepsilon$,

\begin{align}\label{thrdp:eq7c}
\mathbb{P}\left(\sqrt{\sum_{j=j_\varepsilon+1}^\infty d_j^2(\bar{Y}_n-\hat{y},u_j)^2} >\left(\frac{\gamma}{16d^{-\nu}}\right)^\frac{1}{\nu'+1} \sqrt{n}^{-1}\right)&\le \left(\frac{16d^{-\nu}}{\gamma}\right)^\frac{2}{\nu'+1}n\sum_{j=j_\varepsilon+1}^\infty d_j^2\E(\bar{Y}_n-\hat{y},u_j)^2\\\notag
 &= \left(\frac{16d^{-\nu}}{\gamma}\right)^\frac{2}{\nu'+1}\sum_{j=j_\varepsilon+1}^\infty d_j^2\E(Y_1-\hat{y},u_j)^2\\
 & \le \left(\frac{16d^{-\nu}}{\gamma}\right)^\frac{2}{\nu'+1} \sum_{j=j_\varepsilon+1}^\infty j^{-(1+\varepsilon_2)}\le \varepsilon.
\end{align}

Therefore, by \eqref{thrdp:eq7a}, \eqref{thrdp:eq7b} and \eqref{thrdp:eq7c} there holds $\mathbb{P}\left(\Omega_{n,\varepsilon}\right)\ge 1-3\varepsilon$ for $n$ large enough and thus \eqref{thrdp:eq7a1}. Since $\varepsilon>0$ was arbitrary it follows that

\begin{equation}\label{case1data}
\mathbb{P}\left( \sqrt{\sum_{j=1}^{k_n}\frac{(\bar{Y}_n-\hat{y},u_j)^2}{\sigma_j^2}} \le \frac{L}{2}\rho^\frac{1}{\nu'+1}\left(\frac{1}{\sqrt{n}}\right)^\frac{\nu'}{\nu'+1}\right)\to 1
\end{equation}

as $n\to\infty$. Finally, \eqref{case1approx} and \eqref{case1data} together prove Theorem \ref{opt2} for the case, that for all $k\in\N$ there is $j_k\ge k$ with $(\hat{y},u_{j_k})\neq 0$.

\subsubsection{Case 2}
Now we assume, that there exists $J\in\N$ such that $(\hat{y},u_j)=0$ for all $j\ge J$.
 We cannot expect a result similar to Lemma \ref{chapt1:lem2} (since $k_n$ will not converge to $\infty$ in probability), but the true solution $\hat{x}$ has arbitrarily large smoothness. Let $\varepsilon>0$ be such that $(d_J\sigma_J)^{-\varepsilon}\le 2$. We set $\nu''=\nu'+\varepsilon$ and use the representation from  \eqref{opt2:eq1}

 \begin{align}
\hat{x} &= \sum_{j=1}^\infty \sigma_j^\nu (\xi,v_j) = \sum_{j=1}^K (d_j\sigma_j)^{\nu''} (d_j\sigma_j)^{-\varepsilon}\frac{\sigma_j^{\nu-\nu'}}{d_j^{\nu'}}(\xi,v_j)= \sum_{j=1}^\infty (d_j\sigma_j)^{\nu''} (\xi'',v_j)v_j,
\end{align}

with $\xi'':=\sum_{j=1}^\infty (d_j\sigma_j)^{-\varepsilon}\frac{\sigma_j^{\nu-\nu'}}{d_j^{\nu'}}(\xi,v_j)v_j$ and $\|\xi''\| \le 2d^{-\nu}\rho =:\rho''$. We denote 

\begin{equation}\label{thrdp:eq8}
\Omega_n:=\left\{\frac{d_j}{2} \le d_{j,n}\le 2 d_j,~\forall l\le m_n,~\sqrt{\sum_{j=1}^{m_n}d_j^2(\bar{Y}_n-\hat{y},u_j)^2} \le n^\frac{\varepsilon'-1}{2},~ {\delta_n^{est}}'\le n^\frac{\varepsilon'-1}{2}\right\},
\end{equation}

 with $\varepsilon'=\frac{1}{2}\left(\frac{\nu''}{\nu''+1}-\frac{\nu'}{\nu'+1}\right)$. It holds that

\begin{equation}\label{thrdp:eq8a}
\mathbb{P}\left(\Omega_n\right)\to1
\end{equation}

as $n\to\infty$ because of \eqref{thrdp:eq2}, \eqref{thrdp:eq1} and 

\begin{align*}
\mathbb{P}\left(\sqrt{\sum_{j=1}^{m_n}d_j^2(\bar{Y}_n-\hat{y},u_j)^2} > n^\frac{\varepsilon'-1}{2}\right)
&\le n^{-\varepsilon'} \sum_{j=1}^{m_n}d_j^2\E(Y_1-\hat{y},u_j)^2\\
& \le n^{-\varepsilon'} \sum_{j=1}^{\infty} j^{-(1+\varepsilon_2)} \to 0
\end{align*}

as $n\to\infty$. For $n$ large enough (such that $m_n\ge J$) the approximation and discretisation error is
 
 \begin{align*}
 \sqrt{\sum_{j=k_n+1}^{\infty}(\hat{x},v_j)^2}\chi_{\Omega_n}&=\sqrt{\sum_{j=k_n+1}^{J}(\hat{x},v_j)^2}\chi_{\Omega_n}= \sqrt{\sum_{j=k_n+1}^J \frac{d_{j,n}^2(\hat{y},u_j)^2}{d_{j,n}^2\sigma_j^2}}\chi_{\Omega_n}\\
 &\le \frac{1}{d_{J,n}\sigma_J}\left(\sqrt{\sum_{j=k_n+1}^{J} d_{j,n}^2(\bar{Y}_n,u_j)^2} + \sqrt{\sum_{j=k_n+1}^{J} d_{j,n}^2(\bar{Y}_n-\hat{y},u_j)^2}\right)\chi_{\Omega_n}\\
  &\le \frac{1}{d_{J,n}\sigma_J}\left(\sqrt{\sum_{j=k_n+1}^{m_n} d_{j,n}^2(\bar{Y}_n,u_j)^2} + \sqrt{\sum_{j=1}^{m_n} d_{j,n}^2(\bar{Y}_n-\hat{y},u_j)^2}\right)\chi_{\Omega_n}\\
 &\le \frac{2}{d_J\sigma_J}\left({\delta_n^{est}}' + 2\sqrt{\sum_{j=1}^{m_n} d_j^2(\bar{Y}_n-\hat{y},u_j)^2}\right)\chi_{\Omega_n}\le 2^{\frac{1}{\varepsilon}+1}\left(n^{\frac{\varepsilon'-1}{2}} + 2n^{\frac{\varepsilon'-1}{2}}\right)\\
 &\le 2^{\frac{1}{\varepsilon}+3} n^\frac{\varepsilon'-1}{2}\le \frac{L}{2} \rho^\frac{1}{\nu'+1}\left(\frac{1}{\sqrt{n}}\right)^\frac{\nu'}{\nu'+1}
 \end{align*}

for $n$ large enough, where we used $d_J\sigma_J\ge 2^{-\frac{1}{\varepsilon}}$ in the sixth step, the definition of the discrepancy principle in the fifth step and 

$$\varepsilon'-1 = -\frac{\nu'}{\nu'+1} + \frac{1}{2}\frac{\nu''}{\nu''+1} + \frac{1}{2}\frac{\nu'}{\nu'+1} - 1 > -\frac{\nu'}{\nu'+1}+\frac{1}{2}+\frac{1}{2}-1=-\frac{\nu'}{\nu'+1}$$

 in the last step. We therefore obtain
 
 \begin{equation}\label{case2approx}
 \mathbb{P}\left(\sqrt{\sum_{j=k_n+1}^\infty(\hat{x},v_j)^2}\le\frac{L}{2}\rho^\frac{1}{\nu'+1}\left(\frac{1}{\sqrt{n}}\right)^\frac{\nu'}{\nu'+1}\right)\ge \mathbb{P}\left(\Omega_n\right)\to 1
 \end{equation}

as $n\to\infty$.
It remains to treat the data propagation error. We set $b_n:=\left(\frac{\gamma}{8\rho''\sqrt{n}}\right)^\frac{1}{\nu''+1}$. Let the deterministic sequence $(q_n)_{n\in\N}\subset \N$ be defined via

\begin{equation}\label{thrdp:eq8b}
q_n:=\min\left\{j\le m_n~: ~d_j\sigma_j\le b_n\right\}.
\end{equation}

If $b_n> d_j\sigma_j$ for all $j=1,...,m_n$ we set $q_n:=m_n$. Note that $q_n\to\infty$ as $n\to\infty$, since $b_n\to 0$. Define

\begin{align}\label{th4omega1}
\bar{\Omega}_n:&=\left\{\sqrt{\sum_{j=q_n}^{m_n} d_{j}^2(\bar{Y}_n-\hat{y},u_j)^2}< \gamma/\sqrt{8n}\right\}\cap \Omega_n
\end{align}

with $\Omega_n$ from \eqref{thrdp:eq8}. We claim that 

\begin{equation}\label{specialcase:eq1}
k_n\chi_{\bar{\Omega}_n} \le  q_n-1
\end{equation} 
 
as $n\to\infty$. Since \eqref{specialcase:eq1} trivially holds for $q_n=m_n$, we may assume that $d_{q_n}\sigma_{q_n}\le b_n$. Further,

\begin{align*}
\sqrt{\sum_{j=q_n}^{m_n} d_{j,n}^2(\bar{Y}_n,u_j)^2}\chi_{\bar{\Omega}_n} &\le \left(\sqrt{\sum_{j=q_n}^{m_n} d_{j,n}^2 (\hat{y},u_j)^2}+ \sqrt{\sum_{j=q_n}^{m_n} d_{j,n}^2(\bar{Y}_n-\hat{y},u_j)^2}\right)\chi_{\bar{\Omega}_n}\\
 &\le 2\left(\sqrt{\sum_{j=q_n}^{m_n}d_j^2(\hat{y},u_j)^2} + \sqrt{\sum_{j=q_n}^{m_n}d_j^2(\bar{Y}_n-\hat{y},u_j)^2}\right)\chi_{\bar{\Omega}_n}\\
 &\le 2\left(\sqrt{\sum_{j=q_n}^{m_n} (d_j\sigma_j)^{2(\nu''+1)}(\xi'',v_j)^2} + \frac{\gamma}{8\sqrt{n}}\right)\chi_{\bar{\Omega}_n}\\
 &\le 2\left( (d_{q_n}\sigma_{q_n})^{\nu''+1} \rho''+\frac{\gamma}{8\sqrt{n}}\right)\chi_{\bar{\Omega}_n}\\
 &\le \left(2 b_n^{\nu''+1} \rho''+\frac{\gamma}{4\sqrt{n}}\right)\chi_{\bar{\Omega}_n}\le \left(\frac{\gamma}{4\sqrt{n}}+\frac{\gamma}{4\sqrt{n}}\right)\chi_{\bar{\Omega}_n}
 \le {\delta_n^{est}}'
\end{align*}

and the claim \ref{specialcase:eq1} follows by the definition of $k_n$ in Algorithm 1. It holds that 

\begin{equation}\label{specialcase:eq2}
\mathbb{P}\left(\bar{\Omega}_n\right)\to 1
\end{equation}
as $n\to\infty$ because of \eqref{thrdp:eq8a} and

\begin{align*}
\mathbb{P}\left(\sqrt{\sum_{j=q_n}^{m_n} d_j^2(\bar{Y}_n-\hat{y},u_j)^2} >\frac{\gamma}{\sqrt{8 n}}\right) \le \frac{8n}{\gamma^2} \sum_{j=q_n}^{m_n} d_j^2\E(\bar{Y}_n-\hat{y},u_j)^2 \le \frac{8}{\gamma^2}\sum_{j=q_n}^{m_n} j^{-(1+\varepsilon_2)} \to 0 
\end{align*}

as $n\to\infty$. Finally, 
\begin{align*}
\sqrt{\sum_{j=1}^{k_n}\frac{(\bar{Y}_n-\hat{y},u_j)^2}{\sigma_j^2}}\chi_{\bar{\Omega}_n}&= \sqrt{\sum_{j=1}^{k_n} \frac{d_{j}^2(\bar{Y}_n-\hat{y},u_j)^2}{d_{j}^2\sigma_j^2}}\chi_{\bar{\Omega}_n} \le \frac{1}{d_{k_n}\sigma_{k_n}}\sqrt{\sum_{j=1}^{k_n}d_{j}^2(\bar{Y}_n-\hat{y},u_j)^2}\chi_{\bar{\Omega}_n}\\
&\le \frac{1}{b_n} \sqrt{\sum_{j=1}^{m_n}d_j^2(\bar{Y}_n-\hat{y},u_j)^2}\chi_{\bar{\Omega}_n} \le \frac{1}{b_n} n^\frac{\varepsilon'-1}{2}\le \left(\frac{8\rho''}{\gamma}\right)^\frac{1}{\nu''+1}\sqrt{n}^{\frac{1}{\nu''+1}+\varepsilon'-1}\\
&\le \left(\frac{16d^{-\nu}\rho}{\gamma}\right)^\frac{1}{\nu''+1}\sqrt{n}^{-\frac{1}{2}\left(\frac{\nu''}{\nu''+1}+\frac{\nu'}{\nu'+1}\right)}\le \frac{L}{2} \rho^\frac{1}{\nu'+1}\left(\frac{1}{\sqrt{n}}\right)^\frac{\nu'}{\nu'+1},
\end{align*}

for $n$ large enough, where we used $d_{k_n}\sigma_{k_n}\chi_{\bar{\Omega}_n}>b_n\chi_{\bar{\Omega}_n}$ (which follows from $k_n\chi_{\bar{\Omega}_n}\le q_n-1$ and \eqref{thrdp:eq8b}) in the third, the definition of $\Omega_n$ in the fourth, $\rho''=2d^{-\nu}\rho$ and $\varepsilon'=\frac{1}{2}\left(\frac{\nu''}{\nu''+1}-\frac{\nu'}{\nu'+1}\right)$ in the sixth and  $\nu''>\nu'$ in the last step. Thus 

\begin{equation}\label{case2data}
\mathbb{P}\left(\sqrt{\sum_{j=1}^{k_n}\frac{(\bar{Y}_n-\hat{y},u_j)^2}{\sigma_j^2}} \le \frac{L}{2}\rho^\frac{1}{\nu'+1}\left(\frac{1}{\sqrt{n}}\right)^\frac{\nu'}{\nu'+1}\right) \ge \mathbb{P}\left(\bar{\Omega}_n\right)\to 1
\end{equation}

as $n\to\infty$. Both \eqref{case2approx} and \eqref{case2data} together yield the claim of Theorem \ref{opt2} for the case, that there is a $J\in\N$ such that $(\hat{y},u_j)=0$ for all $j\ge J$.

%\end{proof}

\section{Numerical demonstration}\label{sec:4}

We now numerically test the modified discrepancy principle for the toy problem 'deriv2' from the open source MATLAB package Regutools \cite{Hansen:2007}. This is a discretisation of a 1d-Fredholm integral equation by means of the Galerkin approximation with box functions. The resulting discrete problem reads $A\hat{x}=b$, with $\hat{x},b\in\R^m$ and $A\in\R^{m\times m}$. We perturbed the right hand side component wise according to

$$z_i=b_i + \frac{\|b\|}{\sqrt{m}} \delta_i$$ 

where the $\delta_i$ are centralised i.i.d random variables following a generalised Pareto-distribution with finite fourth moment, but infinite higher moments (function gprnd(K,$\sigma$,$\theta$,m,n) with $K=1/5$, $\sigma=\sqrt{(1-K)^2(1-2K)}$ and $\theta=0$). We consider the symmetrised equation (as in Example \ref{int:ex1}) and set $K:=A^*A$ with i.i.d measurements $Y_1, Y_2,...$ distributed as 

$$Y_1 \stackrel{d}{=} A^*\begin{pmatrix} z_1 \\ ... \\ z_m \end{pmatrix}.$$

% Note that the components $(Y_1,v_1),(Y_1,v_2),...$ are not independent.
 We verify that the condition for the fourth moments in Theorem \ref{opt2} is satisfied. Indeed, it holds that

\begin{align*}
&\sup_{j} \frac{\E[\left(Y_1-\hat{y},v_j\right)^4]}{(\E[(Y_1-\hat{y},v_j)^2])^2}\\
 = &\sup_{j} \frac{\E\left[\left(A^*\begin{pmatrix} \delta_1\\\delta_2\\...\end{pmatrix},v_j\right)^4\right]}{\left(\E\left[\left(A^*\begin{pmatrix}\delta_1\\\delta_2\\...\end{pmatrix},v_j\right)^2\right]\right)^2} =\sup_{j} \frac{\E\left[\left(\begin{pmatrix} \delta_1\\\delta_2\\...\end{pmatrix},u_j\right)^4\right]}{\left(\E\left[\left(\begin{pmatrix}\delta_1\\\delta_2\\...\end{pmatrix},u_j\right)^2\right]\right)^2}\\
                    =&\sup_j \frac{\E\left[\left( \sum_{l}\delta_l(e_l,u_j)\right)^4\right]}{\left(\E\left[\left(\sum_{l}\delta_l(e_l,u_j)\right)^2\right]\right)^2} = \sup_j \frac{\E[\delta_1^4]\sum_{l}(e_l,u_j)^4 + 3\left(\E [\delta_1^2]\right)^2 \sum_{\substack{l,l'\\l\neq l'}}(e_l,u_j)^2(e_{l'},u_j)^2}{\left(\E[\delta_1^2]\sum_{l}(e_l,u_j)^2\right)^2}\\
                    \le& \frac{\E[\delta_1^4]}{\left(\E[\delta_1^2]\right)^2}\sup_j \frac{\sum_l(e_l,u_j)^4 + 3 \sum_{\substack{l,l'\\l\neq l'}} (e_l,u_j)^2(e_{l'},u_j)^2}{\sum_l(e_l,u_j)^4 + \sum_{\substack{l,l'\\l\neq l'}}(e_l,u_j)^2(e_{l'},u_j)^2} \le 4 \frac{\E[\delta_1^4]}{(\E[\delta_1^2])^2},
\end{align*}

where $e_1,e_2,...$ is the (orthonormal) Galerkin basis and $\hat{y} = A^*b$. We set the discretiation to $m=1000$ and approximated the singular value decomposition $(\sigma_j,u_j,v_j)$ of $A$  with the function 'csvd'. We used $n=[50, 500, ...,500000]$ measurements and compared the classical discrepancy principle to the modified one implemented in Algorithm 1 with $\varepsilon_1=0.5$ and $\varepsilon_2=0.5$ (large) and $\varepsilon_2=0.1$ (small). We calculated the relative errors for $100$ independent runs and visualised the results as box plots in Figure 1. We clearly see that the errors decay faster for the modified discrepancy principle. Moreover, in Table \ref{tab1} we compare the (relative) median error of the plain and modified discrepancy principle (this is the red bar in each of the boxes) to the square root of the minimax risk from Theorem $\ref{opt1}$, where we sampled the latter from the same data. We see that the error of the modified discrepancy principle is comparable to the minimax risk for smaller sample sizes. For larger sample sizes the minimax risk is better, which is consistent with the loss of $\varepsilon_2$ in the exponent of \eqref{int:eq1ba}.
\begin{figure}\label{fig11}
\centering
\includegraphics[width=0.9\linewidth]{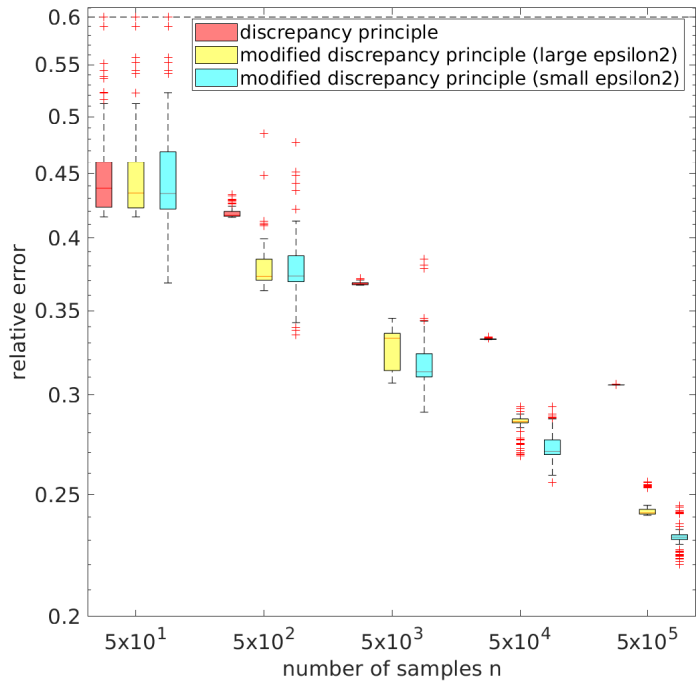} 
\caption{Relative errors for (modified) discrepancy principle. We see that the resulting errors decay faster for the modified discrepancy principle.}
\end{figure}

\begin{table}[hbt!]
\centering
\caption{Sampled median of the relative error of the plain and modified discrepancy principle (dp) and sampled relative minimax risk for different sample sizes.}\label{tab1}
\setlength{\tabcolsep}{4pt}
\begin{tabular}{c|ccc|c|}
\toprule
\multicolumn{1}{c}{sample size} & \multicolumn{1}{c}{plain dp} & \multicolumn{1}{c}{modified dp, $\varepsilon_2$ large} & \multicolumn{1}{c}{modified dp, $\varepsilon_2$ small} & \multicolumn{1}{c}{oracle}  \\
\cmidrule(l){2-2} \cmidrule(l){3-3} \cmidrule(l){4-4} \cmidrule(l){5-5}
$n$ & median error & median error & median error & square root of the minimax risk \\
5e1 &  4.36e-1 &  4.36e-1 &  4.4e-1  & 4.48e-1 \\
5e2 &  4.17e-1 &  3.75e-1 &  3.72e-1  & 3.69e-1 \\
5e3 &  3.67e-1 &  3.32e-1 &  3.14e-1  & 3.07e-1 \\
5e4 &  3.32e-1 &  2.85e-1 &  2.69e-1  & 2.54e-1  \\
5e5 &  3.05e-1 &  2.41e-1 &  2.32e-1 & 2.09e-1 \\
\bottomrule
\end{tabular}
\end{table}

\section{Concluding remarks}\label{sec:5}

In this work we have presented a modified discrepancy principle, which yields (almost) optimal convergence rates for arbitrary unknown error distributions, if one is able to repeat the measurements. This was achieved in estimating the variances of one measurement along the singular directions of the operator $K$, which was then used to rescale the measurements and the operator.

We restricted to linear mildly ill-posed problems and classical H\"older-type source conditions in Hilbert spaces, but the results probably can be extended to general degree of ill-posedness and general source conditions. A major drawback is, that the singular value decomposition of the operator needs to be known. It would be interesting to investigate whether the approach could be adapted to settings, where the singular valued decomposition is not given. 

\section*{Acknowledgements}
The author would like to thank Dr. Peter Math\'e for kindly hosting him in Berlin, where we worked out the basic idea of the presented approach.

\bibliographystyle{abbrv}
\bibliography{references}
\end{document}